\documentclass[openany,a4paper,12pt]{article}
\usepackage[utf8x]{inputenc}
\usepackage[T1]{fontenc}
\usepackage[english]{babel}
\usepackage[intlimits]{amsmath}
\usepackage[]{graphicx}
\usepackage{color}
\usepackage{hyperref} 
\usepackage{amssymb}
\usepackage{amsmath}
\usepackage{array}
\usepackage{geometry}
\usepackage{pdfpages}

\newcommand{\K}{\mathbb{K}}
\newcommand{\C}{\mathbb{C}}

\newtheorem{defi}{Definition}
\newtheorem{con}[defi]{Conjecture}
\newtheorem{teor}[defi]{Theorem}
\newtheorem{lema}[defi]{Lemma}

\newenvironment{proof}{{\it Proof:}}{\hspace{\stretch{1}}\rule{1ex}{1ex}}

\usepackage{titlesec}

\titleformat*{\section}{\normalsize\bfseries}
\titleformat*{\subsection}{\normalsize\bfseries}
\titleformat*{\subsubsection}{\normalsize\bfseries}
\titleformat*{\paragraph}{\normalsize\bfseries}
\titleformat*{\subparagraph}{\normalsize\bfseries}

\usepackage{setspace}
\begin{document}

\begin{center}
{\normalsize \bf THE IMAGES OF MULTILINEAR POLYNOMIALS ON STRICTLY UPPER TRIANGULAR MATRICES}\\
{\normalsize Pedro S. Fagundes}\footnote{pedrosfmath@gmail.com\\ Universidade Federal de São Paulo, Instituto de Ciência de Tecnologia, SP, Brazil}
\end{center}

\begin{abstract}
The purpose of this paper is to describe the images of multilinear polynomials of arbitrary degree on the strictly upper triangular matrix algebra. 

\noindent
{\bf Key words:} Lvov-Kaplansky conjecture, multilinear polynomials, strictly upper triangular matrices.
\end{abstract}

\section{Introduction}
Let $\K$ be any field and let $\K\langle X \rangle$ be the free associative algebra over $\K$, freely generated by  the countable set $X=\{x_{1},x_{2},\dots\}$ of noncommuting variables.

Our main motivation is an old problem due to Lvov \cite{Dniester} (which is also attributed to Kaplansky \cite{Kanel2}):

\begin{con}\label{c1}
The image of a multilinear polynomial in $\K\langle X \rangle$ on the matrix algebra $M_{n}(\K)$ is a vector space. 
\end{con}

This conjecture was inspired by classical results due to Shoda \cite{Shoda} and Albert and Muckenhoupt  \cite{Albert} where it was verified for polynomials of degree two. 

In case of multilinear polynomials of degrre three over the complex number field $\C$, Dykema and Klep \cite{Dykema} verified Conjecture \ref{c1} when $n$ is even or $n<17$.

In 2013, Mesyan  \cite{Mesyan} found an important relation between images of multilinear polynomials $f\in\K\langle X \rangle$ of degree three on $M_{n}(\K)$ and the traceless matrix algebra $sl_{n}(\K)$. He showed, under some mild condition on $\K$, that $sl_{n}(\K)\subset f(M_{n}(\K))$, where $f(M_{n}(\K))$ denotes the image of $f$ on $M_{n}(\K)$. In his paper, Mesyan posed the following problem.

\begin{con}\label{c2}
Let $n\geq2$ and $m\geq1$ be integers, let $f(x_{1},\dots,x_{m})\in \K\langle X \rangle$ be a nonzero multilinear polynomial and let $n\geq m-1$. Then $sl_{n}(\K)\subset f(M_{n}(\K))$.
\end{con}

In 2013 the Mesyan's Conjecture was positively answered by Buzinski and Winstanley  \cite{Buzinski} for polynomials of degree four over algebraically closed fields of characteristic zero.

A major breakthrough in Conjecture \ref{c1} was done in 2012 by Kanel-Belov, Malev and Rowen  \cite{Kanel2}, when they solved it for $n=2$, in case of quadratically closed field. Some further developments of their approach can be found in \cite{Kanel3, Kanel4, Kanel5,Malev}.

In attempt of approaching the Lvov-Kaplansky conjecture, some variations of it have been studied. For example, the images of multilinear polynomials of small degree on Lie Algebras  (\cite{Anzis}, \cite{Spela}), Jordan Algebras \cite{Ma} and on the upper triangular matrix algebra  \cite{Fagundes} were described.

The main goal of this paper is to discuss another variation of Conjecture \ref{c1}, namely, the description of the image of a multilinear polynomial on strictly upper triangular matrices. But before the statement of the main theorem, we introduce some notations.

From now on, $\K$ will denote an arbitrary field. For each $n\geq2$ and $m\geq1$, we will denote by $UT_{n}^{(m-1)}(\K)$ (or simply by $UT_{n}^{(m-1)}$) the subalgebra of the upper triangular matrix algebra $UT_{n}(\K)$ whose $(p,q)$ entry is zero when $q-p\leq m-1$. In other words, the matrices in $UT_{n}^{(m-1)}$ are such that the $m$ first diagonals are all null. We note that for $m=1$, $UT_{n}^{(0)}$ is the strictly upper triangular matrix algebra.
 
Hence, our theorem is

\begin{teor}\label{t1}
Let $\K$ be any field, let $n\geq2$ and $m\geq1$ be integers. Let $f(x_{1},\dots,x_{m})\in\K\langle X \rangle$ be a nonzero multilinear polynomial. Then the image of $f$ on $UT_{n}^{(0)}$ is either $\{0\}$ or $UT_{n}^{(m-1)}$.
\end{teor} 

We can assume that $m\geq2$, because for polynomials of degree 1 the statement is obvious.

Writing $$f(x_{1},\dots,x_{m})=\displaystyle\sum_{\sigma\in S_{m}}\lambda_{\sigma}x_{\sigma(1)}\cdots x_{\sigma(m)},$$
we will also assume, without loss of generality, that $\lambda_{id}=1$, where $S_{m}$ is the symmetric group of the set $\{1,\dots,m\}$ and $\lambda_{\sigma}\in\K$.

Since $UT_{n}^{(0)}$ is a nilpotent algebra of index $n$, any nonzero multilinear polynomial $f(x_{1},\dots,x_{m})$ is a polynomial identity for $UT_{n}^{(0)}$, when $m\geq n.$

Observe that $f(x_{1},\dots,x_{m})$ is not a polynomial identity when $n>m$, since replacing $x_{j}$ by $e_{j,j+1}$ we have 
\begin{eqnarray}\nonumber
f(e_{1,2},\dots, e_{m,m+1})=e_{1,m+1}\neq0.
\end{eqnarray}

Before the proof of Theorem \ref{t1}, we state some technical results.

\section{Some technical lemmas}

Let $Y=\{y_{k}^{(l)}|k,l\in\{1,\dots,n\}\}$ be a set of commuting variables. 
It is convenient for us to use both subscript and superscript indices for these
variables because it will be easier to see how we act on the superscript indices by permutations. 

Let $\K[Y]$ be the algebra of polynomials on $Y$ over a field $\K$.

For each $j\in\{1,\dots,m\}$, we will consider the following notations: 
$$S_{m}^{(j)}=\{\sigma\in S_{m}|\sigma(j)=j\} \ \mbox{and} \ G^{(j)}=S_{m}^{(1)}\cap S_{m}^{(j)}\cap S_{m}^{(j+1)}\cap\dots\cap S_{m}^{(m)}.$$
We will also denote $S_{m}^{(1)}$ by $G^{(m+1)}$.

The main goal of this section is to prove the next lemma, which plays a key role in the proof of Theorem \ref{t1}.

\begin{lema}\label{l1}
Let $\sigma\in S_{m}^{(1)}$ and $\lambda_{\sigma}\in\K$ where $\lambda_{id}=1$. Then we can replace the variables $y_{2}^{(2)},\dots,y_{n-1}^{(2)},\dots,y_{2}^{(m)},\dots,y_{n-1}^{(m)}$ by scalars in $\K$ such that all the following polynomials
\begin{eqnarray}\label{a4}
\left\{\begin{array}{c}
\displaystyle\sum_{\sigma\in S_{m}^{(1)}}\lambda_{\sigma}y_{2}^{(\sigma(2))}\cdots y_{m}^{(\sigma(m))}\\
\displaystyle\sum_{\sigma\in S_{m}^{(1)}}\lambda_{\sigma}y_{3}^{(\sigma(2))}\cdots y_{m+1}^{(\sigma(m))}\\
\vdots \\
\displaystyle\sum_{\sigma\in S_{m}^{(1)}}\lambda_{\sigma}y_{n-m+1}^{(\sigma(2))}\cdots y_{n-1}^{(\sigma(m))}
\end{array}\right.
\end{eqnarray}
take some nonzero values in $\K$.
\end{lema}

We will divide the proof of the previous lemma in the next ones.

\begin{lema}\label{l2}
Let $\sigma\in S_{2}^{(1)}$ and $\lambda_{\sigma}\in\K$ where $\lambda_{id}=1$. Then we can replace the variables $y_{2}^{(2)},\dots,y_{n-1}^{(2)}$ by scalars in $\K$ such that all the following polynomials
\begin{eqnarray}\label{a3}
\left\{\begin{array}{c}
\displaystyle\sum_{\sigma\in S_{2}^{(1)}}\lambda_{\sigma}y_{2}^{(\sigma(2))}\\
\displaystyle\sum_{\sigma\in S_{2}^{(1)}}\lambda_{\sigma}y_{3}^{(\sigma(2))}\\
\vdots \\
\displaystyle\sum_{\sigma\in S_{2}^{(1)}}\lambda_{\sigma}y_{n-1}^{(\sigma(2))}
\end{array}\right.
\end{eqnarray}
take some nonzero values in $\K$.
\end{lema}

\begin{proof}
Since $S_{2}^{(1)}=\{id\}$, each polynomial in (\ref{a3}) can be written as $\displaystyle\sum_{\sigma\in S_{2}^{(1)}}\lambda_{\sigma}y_{k+1}^{(\sigma(2))}=y_{k+1}^{(2)}$ with $k\in\{1,\dots,n-2\}$. Then we replace each $y_{k+1}^{(2)}$ by $1$ and get nonzero values.
\end{proof}

\begin{lema}\label{l3}
Let $\sigma\in S_{3}^{(1)}$ and $\lambda_{\sigma}\in\K$ where $\lambda_{id}=1$. Then we can replace the variables $y_{2}^{(2)},\dots,y_{n-1}^{(2)},y_{2}^{(3)},\dots,y_{n-1}^{(3)}$ by scalars in $\K$ such that all the following polynomials
\begin{eqnarray}\label{a1}
\left\{\begin{array}{c}
\displaystyle\sum_{\sigma\in S_{3}^{(1)}}\lambda_{\sigma}y_{2}^{(\sigma(2))}y_{3}^{(\sigma(3))}\\
\displaystyle\sum_{\sigma\in S_{3}^{(1)}}\lambda_{\sigma}y_{3}^{(\sigma(2))}y_{4}^{(\sigma(3))}\\
\vdots\\
\displaystyle\sum_{\sigma\in S_{3}^{(1)}}\lambda_{\sigma}y_{n-2}^{(\sigma(2))}y_{n-1}^{(\sigma(3))}
\end{array}\right.
\end{eqnarray}
take some nonzero values in $\K$.
\end{lema}

\begin{proof}
Since $S_{3}^{(1)}=\{id,(23)\}$, we can rewrite (\ref{a1}) as  
\begin{eqnarray}\label{a2}
\left\{\begin{array}{c}
y_{2}^{(2)}y_{3}^{(3)}+\lambda_{(23)}y_{2}^{(3)}y_{3}^{(2)}\\
y_{3}^{(2)}y_{4}^{(3)}+\lambda_{(23)}y_{3}^{(3)}y_{4}^{(2)}\\
\vdots\\
y_{n-2}^{(2)}y_{n-1}^{(3)}+\lambda_{(23)}y_{n-2}^{(3)}y_{n-1}^{(2)}
\end{array}\right.
\end{eqnarray}

If $\lambda_{(23)}=0$, then we just replace each $y_{k+1}^{(2)}$ and $y_{k+2}^{(3)}$ by 1 for $k\in\{1,\dots,n-3\}$, and the proof is done.

If $\lambda_{(23)}\neq0$, then we replace
\begin{eqnarray}
\left\{\begin{array}{c}
y_{k}^{(2)} \ \mbox{by} \ 0 \ \mbox{and} \ y_{k}^{(3)}\ \mbox{by} \ 1, \mbox{if $k$ is odd;}\\
y_{k}^{(2)}\ \mbox{by} \ 1 \ \mbox{and} \ y_{k}^{(3)}\ \mbox{by} \ 0, \mbox{if $k$ is even,}
\end{array}\right.
\end{eqnarray}
for $k\in\{2,\dots,n-1\}$. Therefore, each valuation in (\ref{a2}) will be 1 or $\lambda_{(23)}$.
\end{proof}

The veracity of Lemma \ref{l1} for $m=2$ and $m=3$ follows from Lemma \ref{l2} and Lemma \ref{l3}, respectively.

Before the general proof of Lemma \ref{l1} for $m\geq4$, we will illustrate it in the case $m=4$.

\begin{lema}\label{l4}
Let $\sigma\in S_{4}^{(1)}$ and $\lambda_{\sigma}\in\K$ where $\lambda_{id}=1$. Then we can replace the variables $y_{2}^{(2)},\dots,y_{n-1}^{(2)},y_{2}^{(3)},\dots,y_{n-1}^{(3)},y_{2}^{(4)},\dots,y_{n-1}^{(4)}$ by scalars in $\K$ so that all the following polynomials
\begin{eqnarray}\nonumber
\left\{\begin{array}{c}
\displaystyle\sum_{\sigma\in S_{4}^{(1)}}\lambda_{\sigma}y_{2}^{(\sigma(2))}y_{3}^{(\sigma(3))}y_{4}^{(\sigma(4))}\\
\displaystyle\sum_{\sigma\in S_{4}^{(1)}}\lambda_{\sigma}y_{3}^{(\sigma(2))}y_{4}^{(\sigma(3))}y_{5}^{(\sigma(4))}\\
\vdots\\
\displaystyle\sum_{\sigma\in S_{4}^{(1)}}\lambda_{\sigma}y_{n-3}^{(\sigma(2))}y_{n-2}^{(\sigma(3))}y_{n-1}^{(\sigma(4))}\\
\end{array}\right.
\end{eqnarray}
take some nonzero values in $\K$.
\end{lema}

\begin{proof}
First of all, we rewrite the above system as

\begin{eqnarray}\label{el}
\left\{\begin{array}{c}
\displaystyle\bigg(\sum_{\sigma\in G^{(4)}}\lambda_{\sigma}y_{2}^{(\sigma(2))}y_{3}^{(\sigma(3))}\bigg)y_{4}^{(4)}+\sum_{\sigma \in S_{4}^{(1)}-G^{(4)}}\lambda_{\sigma}y_{2}^{(\sigma(2))}y_{3}^{(\sigma(3))}y_{4}^{(\sigma(4))}\\
\displaystyle\bigg(\sum_{\sigma\in G^{(4)}}\lambda_{\sigma}y_{3}^{(\sigma(2))}y_{4}^{(\sigma(3))}\bigg)y_{5}^{(4)}+\sum_{\sigma \in S_{4}^{(1)}-G^{(4)}}\lambda_{\sigma}y_{3}^{(\sigma(2))}y_{4}^{(\sigma(3))}y_{5}^{(\sigma(4))}\\
\vdots \\
\displaystyle\bigg(\sum_{\sigma\in G^{(4)}}\lambda_{\sigma}y_{n-3}^{(\sigma(2))}y_{n-2}^{(\sigma(3))}\bigg)y_{n-1}^{(4)}+\sum_{\sigma \in S_{4}^{(1)}-G^{(4)}}\lambda_{\sigma}y_{n-3}^{(\sigma(2))}y_{n-2}^{(\sigma(3))}y_{n-1}^{(\sigma(4))}\\
\end{array}\right.
\end{eqnarray}

and then the proof will be obtained by the next two steps.

{\bf Step 1:} We claim that for suitable choices of variables, the polynomials $\displaystyle\sum_{\sigma\in G^{(4)}}\lambda_{\sigma}y_{k+1}^{(\sigma(2))}y_{k+2}^{(\sigma(3))}$ take nonzero values in $\K$, for all $k\in\{1,\dots,n-4\}$.

Indeed, since $G^{(4)}=\{id,(23)\}$, using the same idea as in the proof of Lemma \ref{l3}, we can replace the variables $y_{k+1}^{(\sigma(2))},y_{k+2}^{(\sigma(3))}$ by scalars $\alpha_{k+1}^{(\sigma(2))},\alpha_{k+2}^{(\sigma(3))}$ in $\K$ so that $\displaystyle\sum_{\sigma\in G^{(4)}}\lambda_{\sigma}\alpha_{k+1}^{(\sigma(2))}\alpha_{k+2}^{(\sigma(3))}\neq0$, for all $k\in\{1,\dots,n-4\}$.

{\bf Step 2:} We proceed by applying the iterative process in the following cases.

\begin{itemize}
\item[Case 1:] In the first polynomial of (\ref{el}), we treat all variables $y$'s except $y_{4}^{(4)}$ as scalars $\alpha$'s, and then we arrive at a linear function in terms of $y_{4}^{(4)}$:
\begin{eqnarray}\label{e5}
\bigg(\sum_{\sigma\in G^{(4)}}\lambda_{\sigma}\alpha_{2}^{(\sigma(2))}\alpha_{3}^{(\sigma(3))}\bigg)y_{4}^{(4)}+  \sum_{\sigma \in S_{4}^{(1)}-G^{(4)}}\lambda_{\sigma}\alpha_{2}^{(\sigma(2))}\alpha_{3}^{(\sigma(3))}\alpha_{4}^{(\sigma(4))}.
\end{eqnarray}

By Step 1, the coefficient of $y_{4}^{(4)}$ above is nonzero. Then we can replace $y_{4}^{(4)}$ to be equal to some element in $\K$ so that the value of the polynomial (\ref{e5}) is nonzero in $\K$.


\item[Case 2:]  In the second polynomial in (\ref{el}), we treat all variables $y$'s  except $y_{5}^{(4)}$ as scalars $\alpha$'s, and then we arrive at a linear function in terms of $y_{5}^{(4)}$:
\begin{eqnarray}\label{e6}
\bigg(\sum_{\sigma\in G^{(4)}}\lambda_{\sigma}\alpha_{3}^{(\sigma(2))}\alpha_{4}^{(\sigma(3))}\bigg)y_{5}^{(4)}+  \sum_{\sigma \in S_{4}^{(1)}-G^{(4)}}\lambda_{\sigma}\alpha_{3}^{(\sigma(2))}\alpha_{4}^{(\sigma(3))}\alpha_{5}^{(\sigma(4))}.
\end{eqnarray}

By  Step 1, the coefficient of $y_{5}^{(4)}$ above is nonzero. Then we can replace $y_{5}^{(4)}$ to be equal to some element in $\K$ so that the value of the polynomial (\ref{e6}) is nonzero in $\K$.


\item[\vdots]

\item[Case $n-4$:]  In the last polynomial in (\ref{el}), we treat all variables $y$'s  except $y_{n-1}^{(4)}$ as scalars $\alpha$'s, and then we arrive at a linear function in terms of $y_{n-1}^{(4)}$:
\begin{eqnarray}\label{e7}
\bigg(\sum_{\sigma\in G^{(4)}}\lambda_{\sigma}\alpha_{n-3}^{(\sigma(2))}\alpha_{n-2}^{(\sigma(3))}\bigg)y_{n-1}^{(4)}+  \sum_{\sigma \in S_{4}^{(1)}-G^{(4)}}\lambda_{\sigma}\alpha_{n-3}^{(\sigma(2))}\alpha_{n-2}^{(\sigma(3))}\alpha_{n-1}^{(\sigma(4))}.
\end{eqnarray}

By  Step 1, the coefficient of $y_{n-1}^{(4)}$ above is nonzero. Then we can replace $y_{n-1}^{(4)}$ to be equal to some element in $\K$ so that the value of the polynomial (\ref{e7}) is nonzero in $\K$.

\end{itemize}
\end{proof}

Now we are able to prove Lemma \ref{l1}.
\\

{\it Proof of Lemma \ref{l1}:} By Lemmas \ref{l2} and \ref{l3} we may assume $m\geq4$.

Let $k\in\{1,\dots,n-m\}$. 

Then the $k$-th polynomial in (\ref{a4}) is $$g_{k}=\sum_{\sigma\in S_{m}^{(1)}}\lambda_{\sigma}y_{k+1}^{(\sigma(2))}\cdots y_{k+m-1}^{(\sigma(m))}.$$

Now we observe that $g_{k}$ can be written in the following way

\begin{eqnarray}\nonumber 
g_{k}=\bigg(\bigg(\cdots\bigg(\bigg(\bigg(\sum_{\sigma\in G^{(4)}}\lambda_{\sigma}y_{k+1}^{(\sigma(2))}y_{k+2}^{(\sigma(3))}\bigg)y_{k+3}^{(4)}+\sum_{\sigma\in G^{(5)}-S_{m}^{(4)}}\lambda_{\sigma}y_{k+1}^{(\sigma(2))}y_{k+2}^{(\sigma(3))}y_{k+3}^{(\sigma(4))}\bigg)y_{k+4}^{(5)}\\\nonumber
+\sum_{\sigma\in G^{(6)}-S_{m}^{(5)}}\lambda_{\sigma}y_{k+1}^{(\sigma(2))}y_{k+2}^{(\sigma(3))}y_{k+3}^{(\sigma(4))}y_{k+4}^{(\sigma(5))}\bigg)y_{k+5}^{(6)}+\cdots\bigg)y_{k+m-2}^{(m-1)}\\\nonumber
+\sum_{\sigma\in G^{(m)}-S_{m}^{(m-1)}}\lambda_{\sigma}y_{k+1}^{(\sigma(2))}\cdots y_{k+m-2}^{(\sigma(m-1))}\bigg)y_{k+m-1}^{(m)}+\sum_{\sigma\in S_{m}^{(1)}-S_{m}^{(m)}}\lambda_{\sigma}y_{k+1}^{(\sigma(2))}\cdots y_{k+m-1}^{(\sigma(m))}.
\end{eqnarray}

The proof will be done in $m-2$ steps, where the Step 1 is a special case and for each $j\in\{2,\dots,m-2\}$, in Step $j$ we will use the previous steps to conclude that the polynomials
\begin{eqnarray}\nonumber\label{a5}
\bigg(\cdots\bigg(\bigg(\sum_{\sigma\in G^{(4)}}\lambda_{\sigma}y_{k+1}^{(\sigma(2))}y_{k+2}^{(\sigma(3))}\bigg)y_{k+3}^{(4)}+\sum_{\sigma\in G^{(5)}-S_{m}^{(4)}}\lambda_{\sigma}y_{k+1}^{(\sigma(2))}y_{k+2}^{(\sigma(3))}y_{k+3}^{(\sigma(4))}\bigg)y_{k+4}^{(5)}+\cdots\bigg)y_{j+k+1}^{(j+2)}\\
+\sum_{\sigma\in G^{(j+3)}-S_{m}^{(j+2)}}\lambda_{\sigma}y_{k+1}^{(\sigma(2))}\cdots y_{j+k+1}^{(\sigma(j+2))}
\end{eqnarray}
take some nonzero values in $\K$, for all $k\in\{1,\dots,n-m\}$.

{\bf Step 1:}

We claim that for suitable choices of variables, the polynomials $\displaystyle\sum_{\sigma\in G^{(4)}}\lambda_{\sigma}y_{k+1}^{(\sigma(2))}y_{k+2}^{(\sigma(3))}$ take nonzero values in $\K$, for all $k\in\{1,\dots,n-m\}$.

Indeed, since $G^{(4)}=\{id,(23)\}$, using the same idea as in the proof of Lemma \ref{l3}, we can replace the variables $y_{k+1}^{(\sigma(2))},y_{k+2}^{(\sigma(3))}$ by scalars $\alpha_{k+1}^{(\sigma(2))},\alpha_{k+2}^{(\sigma(3))}$ in $\K$ such that $\displaystyle\sum_{\sigma\in G^{(4)}}\lambda_{\sigma}\alpha_{k+1}^{(\sigma(2))}\alpha_{k+2}^{(\sigma(3))}\neq0$, for all $k\in\{1,\dots,n-m\}$.

Now we assume that the  Step $j-1$ is done, and then the  Step $j$ will be the following.

{\bf Step j:} We proceed by applying the iterative process in the following cases.

\begin{itemize}
\item[Case 1:] For $k=1$ in (\ref{a5}), we treat all variables $y$'s except $y_{j+2}^{(j+2)}$ as scalars $\alpha$'s, and then we arrive at a linear function in terms of $y_{j+2}^{(j+2)}$:
\begin{multline}\label{p1}
\bigg(\cdots\bigg(\bigg(\sum_{\sigma\in G^{(4)}}\lambda_{\sigma}\alpha_{2}^{(\sigma(2))}\alpha_{3}^{(\sigma(3))}\bigg)\alpha_{4}^{(4)}+\sum_{\sigma\in G^{(5)}-S_{m}^{(4)}}\lambda_{\sigma}\alpha_{2}^{(\sigma(2))}\alpha_{3}^{(\sigma(3))}\alpha_{4}^{(\sigma(4))}\bigg)\alpha_{5}^{(5)}+\cdots\bigg)y_{j+2}^{(j+2)}\\
+\sum_{\sigma\in G^{(j+3)}-S_{m}^{(j+2)}}\lambda_{\sigma}\alpha_{2}^{(\sigma(2))}\cdots \alpha_{j+2}^{(\sigma(j+2))}
\end{multline}

By  Step $j-1$, the coefficient of $y_{j+2}^{(j+2)}$ above is nonzero. Then we can take $y_{j+2}^{(j+2)}$ to be equal to some element in $\K$ so that the value of the polynomial (\ref{p1}) is nonzero in $\K$.


\item[Case 2:]  For $k=2$ in (\ref{a5}), we treat all variables $y$'s  except $y_{j+3}^{(j+2)}$ as scalars $\alpha$'s, and then we arrive at a linear function in terms of $y_{j+3}^{(j+2)}$:
\begin{multline}\label{p2}
\bigg(\cdots\bigg(\bigg(\sum_{\sigma\in G^{(4)}}\lambda_{\sigma}\alpha_{3}^{(\sigma(2))}\alpha_{4}^{(\sigma(3))}\bigg)\alpha_{5}^{(4)}+\sum_{\sigma\in G^{(5)}-S_{m}^{(4)}}\lambda_{\sigma}\alpha_{3}^{(\sigma(2))}\alpha_{4}^{(\sigma(3))}\alpha_{5}^{(\sigma(4))}\bigg)\alpha_{6}^{(5)}+\cdots\bigg)y_{j+3}^{(j+2)}\\
+\sum_{\sigma\in G^{(j+3)}-S_{m}^{(j+2)}}\lambda_{\sigma}\alpha_{3}^{(\sigma(2))}\cdots \alpha_{j+3}^{(\sigma(j+2))}
\end{multline}
By  Step $j-1$, the coefficient of $y_{j+3}^{(j+2)}$ above is nonzero. Then we can take $y_{j+3}^{(j+2)}$ to be equal to some element in $\K$ so that the value of the polynomial (\ref{p2}) is nonzero in $\K$.


\item[\vdots]

\item[Case $n-m$:] For $k=n-m$  in (\ref{a5}), we treat all variables $y$'s  except $y_{n-m+j+1}^{(j+2)}$ as scalars $\alpha$'s, and then we arrive at a linear function in terms of $y_{n-m+j+1}^{(j+2)}$:
\begin{multline}\label{p3}
\bigg(\cdots\bigg(\bigg(\sum_{\sigma\in G^{(4)}}\lambda_{\sigma}\alpha_{n-m+1}^{(\sigma(2))}\alpha_{n-m+2}^{(\sigma(3))}\bigg)\alpha_{n-m+3}^{(4)}+\sum_{\sigma\in G^{(5)}-S_{m}^{(4)}}\lambda_{\sigma}\alpha_{n-m+1}^{(\sigma(2))}\alpha_{n-m+2}^{(\sigma(3))}\alpha_{n-m+3}^{(\sigma(4))}\bigg)\alpha_{n-m+4}^{(5)}\\
+\cdots\bigg)y_{n-m+j+1}^{(j+2)}+\sum_{\sigma\in G^{(j+3)}-S_{m}^{(j+2)}}\lambda_{\sigma}\alpha_{n-m+1}^{(\sigma(2))}\cdots \alpha_{n-m+j+1}^{(\sigma(j+2))}
\end{multline}

By  Step $j-1$, the coefficient of $y_{n-m+j+1}^{(j+2)}$ above is nonzero. Then we can take $y_{n-m+j+1}^{(j+2)}$ to be equal to some element in $\K$ so that the value of the polynomial (\ref{p3}) is nonzero in $\K$, and then the lemma is proved.
\end{itemize}
 
\hspace{16.7cm}{\scriptsize $\blacksquare$}

\section{Proof of the main theorem}

We start this section with the following definition.

\begin{defi}\label{d1}
Let $\K$ be any field, let $n\geq2$ be an integer and $i\in\{1,\dots,n\}$. We will say that a matrix in $UT_{n}(\K)$ is $(i)$-diagonal if the $(k,k+(i-1))$ entries are the only ones possibly nonzero, with $k=1,\dots,n-i+1$. In other words, an $(i)$-diagonal matrix is one in the form
\begin{eqnarray}\nonumber
\left(\begin{array}{cccccc}
 & & & \alpha_{1,i} & &0 \\
 & & & & \ddots & \\
 & & & & & \alpha_{n-i+1,n}\\ 
 & & & & & \\ 
 0& & & & & \\
\end{array}\right).
\end{eqnarray}
\end{defi}

It is easy to see that every matrix $B\in UT_{n}^{(m-1)}$ can be written as a sum of $(i)$-diagonal matrices, with $i\in \{m+1,\dots,n\}$. Indeed, if $B=(b_{p,q})_{p,q=1}^{n}$, then we write
\begin{eqnarray}\label{e}
B=\sum_{i=m+1}^{n}B_{i}
\end{eqnarray}
where $B_{i}$ is an $(i)$-diagonal matrix whose $(k,k+(i-1))$ entry is equal to $b_{k,k+(i-1)}$, for all $k\in\{1,\dots,n-i+1\}$.

With a slight modification of Definition \ref{d1}, we can also consider $(i)$-diagonal matrices with entries in $\K[Y]$.

To prove Theorem \ref{t1} we assume that $m\geq2$ and that the image of $f(x_{1},\dots,x_{m})=\displaystyle\sum_{\sigma\in S_{m}}\lambda_{\sigma}x_{\sigma(1)}\cdots x_{\sigma(m)}$ on $UT_{n}^{(0)}$ is nonzero. In other words, we assume $n>m\geq2$ and $\lambda_{id}=1$.

We also observe that 
\begin{eqnarray}\label{e1}
f=\sum_{j=1}^{m}f_{j},
\end{eqnarray}

where each $f_{j}$ is the sum of all monomials of $f$ whose $j$-th variable is equal to $x_{1}$.

Taking 
\begin{eqnarray}\nonumber
\displaystyle x_{1}^{(m+1)}=\sum_{k=1}^{n-1}y_{k}^{(m+1)}e_{k,k+1},x_{2}=\sum_{k=1}^{n-1}y_{k}^{(2)}e_{k,k+1},\cdots,x_{m}=\sum_{k=1}^{n-1}y_{k}^{(m)}e_{k,k+1}
\end{eqnarray}

as $(2)$-diagonal matrices with entries in $\K[Y]$, by (\ref{e1}) we have
\begin{eqnarray}\nonumber 
f(x_{1}^{(m+1)},x_{2},\dots,x_{m})=\sum_{k=1}^{n-m}\bigg(y_{k}^{(m+1)}\sum_{\sigma\in S_{m}^{(1)}}\lambda_{\sigma}y_{k+1}^{(\sigma(2))}\cdots y_{k+m-1}^{(\sigma(m))}\\
+y_{k+1}^{(m+1)}\delta_{2}^{(m+1)}(x_{2},\dots,x_{m})+\cdots+ y_{k+m-1}^{(m+1)}\delta_{m}^{(m+1)}(x_{2},\dots,x_{m})\bigg)e_{k,k+m}
\end{eqnarray}

where $y_{k+j-1}^{(m+1)}\delta_{j}^{(m+1)}(x_{2},\dots,x_{m})$ denotes the $(k,k+m)$ entry of the matrix 
$$f_{j}(x_{1}^{(m+1)},x_{2},\dots,x_{m})$$  for $j=2,\dots,m.$ 

Now considering $B_{m+1}=\displaystyle\sum_{k=1}^{n-m}b_{k}^{(m+1)}e_{k,k+m}\in UT_{n}^{(m-1)}$, we seek for a solution of the following nonlinear system:
\begin{eqnarray}\label{e2}
\scriptsize\left\{\begin{array}{ccc}
\displaystyle y_{1}^{(m+1)}\sum_{\sigma\in S_{m}^{(1)}}\lambda_{\sigma}y_{2}^{(\sigma(2))}\cdots y_{m}^{(\sigma(m))} +y_{2}^{(m+1)}\delta_{2}^{(m+1)}(x_{2},\dots,x_{m})+\cdots+ y_{m}^{(m+1)}\delta_{m}^{(m+1)}(x_{2},\dots,x_{m})&=&b_{1}^{(m+1)}\\
\displaystyle y_{2}^{(m+1)}\sum_{\sigma\in S_{m}^{(1)}}\lambda_{\sigma}y_{3}^{(\sigma(2))}\cdots y_{m+1}^{(\sigma(m))} +y_{3}^{(m+1)}\delta_{2}^{(m+1)}(x_{2},\dots,x_{m})+\cdots+ y_{m+1}^{(m+1)}\delta_{m}^{(m+1)}(x_{2},\dots,x_{m})&=&b_{2}^{(m+1)}\\
&\vdots & \\
\displaystyle y_{n-m}^{(m+1)}\sum_{\sigma\in S_{m}^{(1)}}\lambda_{\sigma}y_{n-m+1}^{(\sigma(2))}\cdots y_{n-1}^{(\sigma(m))} +y_{n-m+1}^{(m+1)}\delta_{2}^{(m+1)}(x_{2},\dots,x_{m})+\cdots+ y_{n-1}^{(m+1)}\delta_{m}^{(m+1)}(x_{2},\dots,x_{m})&=&b_{n-m}^{(m+1)}\\
\end{array}\right.
\end{eqnarray}

Using Lemma \ref{l1}, we can find  matrices in $UT_{n}^{(0)}$ $$\bar{x}_{2}=\sum_{l=1}^{n-1}\alpha_{l}^{(2)}e_{l,l+1},\dots,\bar{x}_{m}=\sum_{l=1}^{n-1}\alpha_{l}^{(m)}e_{l,l+1},$$ so that $\displaystyle\sum_{\sigma\in S_{m}^{(1)}}\lambda_{\sigma}\alpha_{k+1}^{(\sigma(2))}\cdots \alpha_{k+m-1}^{(\sigma(m))}$ are nonzero, for all $k\in\{1,\dots,n-1\}$. Then the system (\ref{e2}) turned into a linear system in the variables $y_{k}^{(m+1)}$, with $k\in\{1,\dots,n-1\}$. 

This system can be solved recursively starting with the last equation: we replace $y_{n-m+1}^{(m+1)}$, $y_{n-m+2}^{(m+1)}$, $\dots$, $y_{n-1}^{(m+1)}$ by any values (for example by $0$), and solve it for $y_{n-m}^{(m+1)}$. Then we solve the previous equation for $y_{n-m-1}^{(m+1)}$ etc.

Hence, any $(m+1)$-diagonal matrix $B_{m+1}$ can be realized as $f(\bar{x}_{1}^{(m+1)},\bar{x}_{2},\dots,\bar{x}_{m})$, for some  matrix $\bar{x}_{1}^{(m+1)}$ in $UT_{n}^{(0)}$.

Now for each $i\in\{m+2,\dots,n\}$, we consider the matrix
\begin{eqnarray}\nonumber
x_{1}^{(i)}=\sum_{k=1}^{n-i+m}y_{k}^{(i)}e_{k,k+i-m}.
\end{eqnarray}
with entries in $\K[Y]$.

Then
\begin{eqnarray}\nonumber 
f(x_{1}^{(i)},\bar{x}_{2},\dots,\bar{x}_{m})=\sum_{k=1}^{n-i+1}\bigg(y_{k}^{(i)}\sum_{\sigma\in S_{m}^{(1)}}\lambda_{\sigma}\alpha_{k+i-m}^{(\sigma(2))}\cdots \alpha_{k+i-2}^{(\sigma(m))}\\
+y_{k+1}^{(i)}\delta_{2}^{(i)}(\bar{x}_{2},\dots,\bar{x}_{m})+\cdots+ y_{k+m-1}^{(i)}\delta_{m}^{(i)}(\bar{x}_{2},\dots,\bar{x}_{m})\bigg)e_{k,k+i-1}
\end{eqnarray}

where $y_{k+j-1}^{(i)}\delta_{j}^{(i)}(\bar{x}_{2},\dots,\bar{x}_{m})$ denotes the $(k,m+k)$ entry of the matrix 
$$f_{j}(x_{1}^{(i)},\bar{x}_{2},\dots,\bar{x}_{m})$$  for $j=2,\dots,m.$ 

We claim that $\displaystyle\sum_{\sigma\in S_{m}^{(1)}}\lambda_{\sigma}\alpha_{k+i-m}^{(\sigma(2))}\cdots \alpha_{k+i-2}^{(\sigma(m))}$ is nonzero for $k\in\{1,\dots,n-i+1\}$. Indeed, we can rewrite this sum as $\displaystyle\sum_{\sigma \in {S_{m}^{(1)}}}\lambda_{\sigma}\alpha_{k+1}^{(\sigma(2))}\cdots \alpha_{k+m-1}^{(\sigma(m))},$ with $k\in\{i-m,\dots,n-m\}$.

Therefore, for each $i\in\{m+2,\dots,n\}$ and any given $(i)$-diagonal matrix $B_{i}=\displaystyle\sum_{k=1}^{n-i+1}b_{k}^{(i)}e_{k,k+i-1}\in UT_{n}^{(m-1)}$, a solution of the following linear system

\begin{eqnarray}\label{e3}
\scriptsize \left\{\begin{array}{ccc}
\displaystyle y_{1}^{(i)}\sum_{\sigma\in S_{m}^{(1)}}\lambda_{\sigma}\alpha_{i+1-m}^{(\sigma(2))}\cdots \alpha_{i-1}^{(\sigma(m))}+y_{2}^{(i)}\delta_{2}^{(i)}(\bar{x}_{2},\dots,\bar{x}_{m})+\cdots+ y_{m}^{(i)}\delta_{m}^{(i)}(\bar{x}_{2},\dots,\bar{x}_{m})&=&b_{1}^{(i)}\\
\displaystyle y_{2}^{(i)}\sum_{\sigma\in S_{m}^{(1)}}\lambda_{\sigma}\alpha_{i+2-m}^{(\sigma(2))}\cdots \alpha_{i}^{(\sigma(m))}+y_{3}^{(i)}\delta_{2}^{(i)}(\bar{x}_{2},\dots,\bar{x}_{m})+\cdots+ y_{m+1}^{(i)}\delta_{m}^{(i)}(\bar{x}_{2},\dots,\bar{x}_{m})&=&b_{2}^{(i)}\\
&\vdots &\\
\displaystyle y_{n-i+1}^{(i)}\sum_{\sigma\in S_{m}^{(1)}}\lambda_{\sigma}\alpha_{n-m+1}^{(\sigma(2))}\cdots \alpha_{n-1}^{(\sigma(m))}+y_{n-i+2}^{(i)}\delta_{2}^{(i)}(\bar{x}_{2},\dots,\bar{x}_{m})+\cdots+ y_{n-i+m}^{(i)}\delta_{m}^{(i)}(\bar{x}_{2},\dots,\bar{x}_{m})&=&b_{n-i+1}^{(i)}\\
\end{array}\right.
\end{eqnarray}

can be found recursively.

So, any $(i)$-diagonal matrix $B_{i}$ can be realized as $f(\bar{x}_{1}^{(i)},\bar{x}_{2},\dots,\bar{x}_{m})$, for some matrix $\bar{x}_{1}^{(i)}$ in $UT_{n}^{(0)}$.

Now given any matrix $B\in UT_{n}^{(m-1)}$, by (\ref{e}) we have
\begin{eqnarray}\nonumber
B=\sum_{i=m+1}^{n}B_{i}=\sum_{i=m+1}^{n}f(\bar{x}_{1}^{(i)},\bar{x}_{2},\dots,\bar{x}_{m})=f(\sum_{i=m+1}^{n} \bar{x}_{1}^{(i)},\bar{x}_{2},\dots,\bar{x}_{m})\in f(UT_{n}^{(0)}).
\end{eqnarray}

Therefore, $UT_{n}^{(m-1)}\subset f(UT_{n}^{(0)})$. Since the other inclusion is trivial, then we get the equality.

\section*{Acknowlegments}

I would like to thank my advisors Dr. Thiago de Mello and Dr. Mikhail Chebotar for the helpful comments and guidance. I also would like to thank the Department of Mathematical Sciences of Kent State University for its hospitality. 

\section*{Funding}

The author was supported by São Paulo Research Foundation (FAPESP), grants nº 2017/16864-5 and nº 2016/09496-7.

\end{document}